\RequirePackage{tikz}
\RequirePackage{stmaryrd}
\documentclass[pdflatex,sn-mathphys]{sn-jnl}

\usepackage{amsthm}
\usepackage{soul}
\usepackage{graphicx}
\usepackage{amssymb}
\usepackage{stmaryrd}
\usepackage{hyperref}
\usepackage{enumerate}
\usepackage{bussproofs}
\hypersetup{
    colorlinks=true,
    linkcolor=blue,
    }
\usetikzlibrary{shapes.geometric, arrows, positioning}
\tikzstyle{box} = [rectangle, rounded corners, minimum width=2cm, minimum height=1cm,text centered, draw=white, fill=white]
\tikzstyle{arrow} = [thick,->,>=stealth]

\newcommand{\al}{\alpha}
\newcommand{\be}{\beta}
\newcommand{\ga}{\gamma}
\newcommand{\Ga}{\Gamma}

\newcommand{\ra}{\rightarrow}
\newcommand{\Ra}{\Rightarrow}

\newcommand{\LR}{\Leftrightarrow}
\newcommand{\mrm}{\mathrm}
\newcommand{\mc}{\mathcal}
\newcommand{\mbf}{\mathbf}

\newcommand{\msf}{\mathsf}

\newcommand{\bsl}{\backslash}
\newcommand{\den}[1]{\llbracket{#1}\rrbracket}

\jyear{2021}%

\theoremstyle{thmstyleone}%
\newtheorem{Theorem}{Theorem}
\newtheorem{Lemma}{Lemma}
\newtheorem{Corollary}{Corollary}
\newtheorem{Proposition}[Theorem]{Proposition}%
\newtheorem{Property}[Theorem]{Property}%

\theoremstyle{thmstylethree}%
\newtheorem{Definition}{Definition}%
\newtheorem{Remark}{Remark}%
\newtheorem{Example}{Example}%

\begin{document}

\title[$Inc\mc{RL}^-$: Logics and Decidability]{Involutive Commutative Residuated Lattice without Unit: Logics and Decidability}


\author[1]{\fnm{Yiheng} \sur{Wang}}\email{ianwang747@gmail.com}

\author[1]{\fnm{Hao} \sur{Zhan}}\email{haozhan1993@gmail.com}

\author[1]{\fnm{Yu} \sur{Peng}}\email{amberlogos@163.com}

\author*[1]{\fnm{Zhe} \sur{Lin}}\email{pennyshaq@163.com}

\affil*[1]{\orgdiv{Department of Philosophy}, \orgname{Xiamen University}, \orgaddress{\street{Siming South Road}, \city{Xiamen}, \postcode{361005}, \state{Fujian}, \country{China}}}




\abstract{We investigate involutive commutative residuated lattices without unit, which are commutative residuated lattice-ordered semigroups enriched with a unary involutive negation operator. The logic of this structure is discussed and the Genzten-style sequent calculus of it is presented.  Moreover, we prove the decidability of this logic.}

\keywords{involutive negation; residuated lattice; unit; decidability}



\maketitle

\section{Introduction}

    The residuated lattice-ordered monoid, or residuated lattice for short, was first introduced by Ward, et al. \cite{krull1924axiomatische,dilworth1938abstract,ward1937residuation} in the early 20th century. A residuated lattice is a structure  $\msf{F}=(A,\land,\vee,\cdot,\bsl,/,1)$ such that $(A,\land,\vee)$ is a lattice, $(A,\cdot,\bsl,/,1)$ is a residuated monoid where $1$ is the unit for fusion ``$\cdot$'', and $\bsl,\cdot,/$ are binary operations satisfying the following law of residuation:
      \[(\mrm{res})\quad a\cdot b\le c\ \text{iff}\ b\le a\bsl c\ \text{iff}\ a\le c/b.\]
    The class of all residuated lattices will be denoted by  $\mc{RL}$. The study of $\mc{RL}$ can be traced back to the work of W. Krull \cite{krull1924axiomatische} and that of Morgan Ward and R. P. Dilworth \cite{ward1937residuation,dilworth1938abstract}. The theory has been studied in several branches, including lattice-ordered groups, substructural logics, and mathematical linguistics. The class of all residuated semigroups is denoted by $\mc{RSG}$. $\mc{RL}$ without the unit $1$ denoted by $\mc{RL^-}$ forms the residuated lattice-ordered semigroup. $\mc{RSG}$, $\mc{RL^-}$ and their non-associative variants are widely used as the algebraic base of categorial grammars for natural language processing (cf. \cite{lambek1958mathematics,kanazawa1992lambek,buszkowski2006lambek}). The commutative $\mc{RL}$ ( $\mc{RL^-}$), denoted by $c\mc{RL}$ ($c\mc{RL^-}$) is $\mc{RL}$ ($\mc{RL^-}$) additionally satisfying the commutativity for ``$\cdot$'' i.e. $a\cdot b=b\cdot a$. Consequently in  $c\mc{RL}$ ($c\mc{RL^-}$), $a\bsl b= b/a$.
    
    There are lots of literatures on $\mc{RL}$ enriched with an involutive negation (in) denoted by $In\mc{RL}$ (cf. \cite{galatos2004adding,jenei2012involutive}). An involutive negation over $\mc{RL}$ and $\mc{RL^-}$ is a unary operation satisfying the following two conditions: (in)  $a\bsl\neg b=\neg a/b$ and (dn) $\neg\neg a= a$ (cf. \cite{galatos2004adding}). Let $0=\neg 1$. Then $a\bsl 0=\neg a/ 1=\neg a$. Thus $\neg a =a \bsl 0 =0/ a$. Therefore the contraposition (ctr) follows: if $a \leq b$, then $\neg b\leq \neg a$. In  $c\mc{RL}$, (in) is equivalent to (cin): $a\bsl\neg b= b\bsl \neg a$. If one replaces (dn) by $a\leq \neg\neg a$, then one obtains the quasi-involutive negation. Hereafter we denote  $Inc\mc{RL}$ and $Inc\mc{RL^-}$ for involutive $c\mc{RL}$ and  $c\mc{RL^-}$ respectively. Evidently, quasi-involutive negation and involutive negation are minimal negation and De morgan negation respectively. De Morgan negation is a unary operation satisfying (ctr) and (dn) (cf. \cite{dunn1999comparative}). If one omits $\neg\neg a\leq a $ from a De Morgan negation, then one obtains the minimal negation. The concept of minimal negation was introduced by Kolmogoroff \cite{Kolmogoroff1927}, but it is systematically studied by Johansson \cite{Johansson1936}. De Morgan algebras (also called ``quasi-Boolean algebras''), which are (not necessarily bounded) distributive lattices with a De Morgan negation, were introduced by Bialynicki-Birula and Rasiowa \cite{bialynicki1957representation}. This type of algebra was also investigated by Moisil \cite{moisil1935recherches} under the term De Morgan lattices, and by Kalman \cite{kalman1958lattices} under the term distributive i-lattices. 

    Substructural logics are the logics whose algebraic models are residuated structures. The research of substructural logics and residuated lattices has been risen by Ono and other people (cf. \cite{schroeder1993substructural,restall2002introduction,buszkowski2006lambek,galatos2007residuated}) over the last four decades. The logic of $\mc{RL}$  under the name full Lambek calculus (FL), is a lattice extension of Lambek calculus (L) with the unit. Lambek calculus is the syntactic calculus introduced by Lambek \cite{lambek1958mathematics} for linguistics analysis. FL without unit denoted by FL$^-$, which happens to be the algebraic model of $\mc{RL^-}$ is introduced by Kanazawa \cite{kanazawa1992lambek}. Its distributive variants (DFL$^-$) are studied by Kozak \cite{kozak2009distributive}. In \cite{kozak2009distributive} the finite model property  and decidability results for DFL$^-$ are proved. Buszkowski considers their nonassociative variants and similar results are established in \cite{buszkowski2011interpolation}. 

    Negative extensions of substructural logics corresponding to  $\mc{RL}$ and its variants are considered in different literatures. $Inc\mc{RL}$ (also named FL$_e$-monoids) are studied by Jenei and Ono \cite{jenei2012involutive}. Galatos and Jipsen \cite{galatos2013residuated} show that $In\mc{RL}$ and its nonassociative variants have the finite model property. Cut-free sequent calculi for the corresponding algebras are presented in the same paper.  Buszkowski \cite{buszkowski2016classical} considers involutive residuated groupoids. The logic of these algebras is called classical nonassociative Lambek Calculus (CNL). Cut-free one side sequent calculus is presented and polynomial decidability is proved for CNL and CNL with the unit (CNL1).

    We continue this line of research. We consider (quasi) involutive negations over $c\mc{RL^-}$. $c\mc{RL^-}$ and $c\mc{RL}$ are essentially different. For instance, the inequation $a/(b/b)\le a$ holds in $c\mc{RL}$ but not in $c\mc{RL^-}$. Moreover $a\cdot \neg a\leq 0$ is a basic law in  $Inc\mc{RL}$. However, this expression lacks sense in $Inc\mc{RL^-}$ without $0$. Further up to the knowledge of the authors, the algebraic properties of $Inc\mc{RL^-}$ are not well investigated. And the logical characterization, cut-free sequent calculus, and the decidability  results of the corresponding logic remain open. Clearly, these results cannot be easily adapted from the ones for $Inc\mc{RL}$. In the present paper, we study $Inc\mc{RL^-}$ and its corresponding logic InFL$^-_e$. A sequent calculus for InFL$^-_e$ is presented and the decidability results are established. Our method is inspired by Lin et al. \cite{lin2021residuated}.

    This paper is organized as follows. Section \ref{section2} gives preliminaries of $c\mc{RL^-}$, $c\mc{RL^-}$ with (quasi) involutive negation and a brief discussion about the relation among different negations. Section \ref{section3} presents the sequent calculus of the corresponding logic, and the results about soundness, completeness, and embeddability. Section \ref{section4} shows a kind of cut-elimination and proves decidability results.
    
\section{Algebra}\label{section2}

    In this section, we consider the commutative residuated lattice-ordered semigroup with an involutive negation operator. Let us recall some basic definitions first.

    \begin{Definition}
    An algebra $\msf{F}=(A,\land,\vee,\cdot,\bsl)$ is called a commutative residuated lattice-ordered semigroup (denoted by $c\mc{RL^-}$) if and only if:
        \begin{itemize}
            \item $(A,\land,\vee)$ is a lattice;
            \item $(A,\cdot)$ is a commutative semigroup;
            \item $\cdot,\bsl$ are operations on $A$ satisfying the following property:
                \begin{itemize}
                    \item[$\mrm{(res)}$] $a\cdot b\le c$ iff $b\le a\bsl c$.
                \end{itemize}
        \end{itemize}
    \end{Definition}

    \begin{Definition}
    An algebra $\msf{F}=(A,\land,\vee,\cdot,\bsl,1)$ is called a commutative residuated lattice (denoted by $c\mc{RL}$) where $(A,\cdot,1)$ is a commutative monoid.
    \end{Definition}

    \begin{Property}\label{propertylattice}
    The following properties related to the lattice operation hold in $c\mc{RL^-}$ and $c\mc{RL}$ for all $a,b,c\in A$:
        \begin{itemize}
            \item[(1)] If $a\leq b$, then $a \land c \leq b $;
            \item[(2)] If $a\leq b$, then $a\leq b\vee c$;
            \item[(3)] If $a\leq b$ and $c\leq b$, then $a\vee c\leq b$;
            \item[(4)] If $a\leq b$ and $a\leq c$, then $a \leq b\land c $;
            \item[(5)] If $a\leq b$ and $b\leq c$, then $a\leq c$.
        \end{itemize}
    \end{Property}

    \begin{Property}\label{propertyfusion}
    The following properties related to the fusion operation hold in $c\mc{RL^-}$ and $c\mc{RL}$ for all $a,b,c\in A$:
        \begin{itemize}
            \item[(1)] $a\cdot b=b\cdot a$;
            \item[(2)] $b\cdot(b\bsl a)\le a$;
            \item[(3)] If $a_1\le b_1$ and $a_2\le b_2$, then $a_1\cdot a_2\le b_1\cdot b_2$;
            \item[(4)] $a\cdot(b\vee c)=(a\cdot b)\vee(a\cdot c)$;
            \item[(5)] If $a\le b$, then $a\cdot c\le b\cdot c$;
            \item[(6)] $(a\bsl c)\cdot b\le a\bsl(c\cdot b)$;
            \item[(7)] $(a\bsl b)\cdot(b\bsl c)\le a\bsl c$;
            \item[(8)] $a\bsl c\le(b\cdot a)\bsl(b\cdot c)$;
            \item[(9)] $b\bsl(a\bsl c)=(a\cdot b)\bsl c$.
        \end{itemize}
    \end{Property}

    \begin{proof}
    We only show the proofs for (3)-(9).
        \begin{itemize}
            \item[(3)] Assume $a_1\le b_1$ and $a_2\le b_2$. Since one has $b_1\cdot a_2\le b_1\cdot a_2$ and $b_1\cdot b_2\le b_1\cdot b_2$, by multiple applications of (res) and the commutativity property, one has $a_1\cdot a_2\le b_1\cdot b_2$.
            \item[(4)] Since $b\le b\vee c$,  by (3) one has $a\cdot b\le a\cdot(b\vee c)$ and $a\cdot c\le a\cdot(b\vee c)$. Therefore, one has $(a\cdot b)\vee(a\cdot c)\le a\cdot(b\vee c)$. Next since $a\cdot b\le(a\cdot b)\vee(a\cdot c)$ and $a\cdot c\le(a\cdot b)\vee(a\cdot c)$, by (res) one has $b\vee c\le a\bsl((a\cdot b),(a\cdot c))$. Again by (res), one has $a\cdot(b\vee c)\le(a\cdot b)\vee(a\cdot c)$. Consequently, one has $a\cdot(b\vee c)=(a\cdot b)\vee(a\cdot c)$.
            \item[(5)] This is a special case of (3).
            \item[(6)] By (2), one has $a\cdot(a\bsl c)\le c$. By (5), one has $(a\cdot(a\bsl c))\cdot b\le c\cdot b$. Therefore, by (res) one has $(a\bsl c)\cdot b\le a\bsl(c\cdot b)$.
            \item[(7)] By application of (6), one has $(a\bsl b)\cdot(b\bsl c)\le a\bsl (b\cdot(b\bsl c))$. Further one has $a\cdot(a\bsl (b\cdot(b\bsl c)))\le c$ by (2). Therefore by (res), one has $(a\bsl b)\cdot(b\bsl c)\le a\bsl c$.
            \item[(8)] The proof is quite similar to (7).
            \item[(9)] Since one has $a\cdot(b\cdot(b\bsl(a\bsl c)))\le c$ and $(a\cdot b)\cdot((a\cdot b)\bsl c)\le c$, then by (res) one has $(a\cdot b)\bsl c\le b\bsl(a\bsl c)$ and $b\bsl(a\bsl c)\le(a\cdot b)\bsl c$. Consequently, one has $b\bsl(a\bsl c)=(a\cdot b)\bsl c$.
        \end{itemize}
    \end{proof}

    \begin{Property}\label{propertyresidual}
    The following properties related to the residual hold in $c\mc{RL^-}$ and $c\mc{RL}$ for all $a,b,c\in A$:
        \begin{itemize}
            \item[(1)] If $a\le b$, then $b\bsl c\le a\bsl c$;
            \item[(2)] If $a\le b$, then $c\bsl a\le c\bsl b$;
            \item[(3)] $c\bsl(a\land b)=(c\bsl a)\land(c\bsl b)$;
            \item[(4)] $b\bsl c\le(a\bsl b)\bsl(a\bsl c)$;
            \item[(5)] $b/a\le(c/b)\bsl(c/a)$ and $a\bsl b\le(a\bsl c)/(b\bsl c)$;
            \item[(6)] $a\bsl(c/b)=(a\bsl c)/b$;
            \item[(7)] $c\le(a/c)\bsl a$ and $c\le a/(c\bsl a)$;
            \item[(8)] $a\bsl b=b/a$.
        \end{itemize}
    \end{Property}

    \begin{proof}
    \ 
        \item[(1)] Assume $a\le b$, one has $b\cdot(b\bsl c)\le c$ by (res). Therefore, one has  $a\le(b\bsl c)\bsl c$. By twice application of (res), one has $b\bsl c\le a\bsl c$.
        \item[(2)] Assume $a\le b$, one has $c\cdot(c\bsl a)\le a$ by (res). Therefore, one has $c\cdot(c\bsl a)\le b$, again by (res) one has $c\bsl a\le c\bsl b$.
        \item[(3)] By application of (res) on $c\bsl(a\land b)\le c\bsl(a\land b)$, one has $c\cdot(c\bsl(a\land b))\le a\land b$. By lattice order one has $c\cdot(c\bsl(a\land b))\le a$, then by (res) one has $c\bsl(a\land b)\le c\bsl a$. Again by a similar method and lattice order one has $c\bsl(a\land b)\le(c\bsl a)\land(c\bsl b)$. For another direction, by application of (res) on $(c\bsl a)\land(c\bsl b)\le c\bsl a$, one has $c\cdot((c\bsl a)\land(c\bsl b))\le a$. Then by (res) again and a similar method one has $(c\bsl a)\land(c\bsl b)\le c\bsl(a\land b)$. Consequently, one has $c\bsl(a\land b)=(c\bsl a)\land(c\bsl b)$.
        \item[(4)] By Property \ref{propertyfusion} (7), one has $(a\bsl b)\cdot(b\bsl c)\le a\bsl c$. Further by application of (res) on $(a\bsl b)\cdot(b\bsl c)\le a\bsl c$, one has $b\bsl c\le (a\bsl b)\bsl(a\bsl c)$.
        \item[(5)] By Property \ref{propertyfusion} (7), one has $(c/b)\cdot(b/a)\le(c/a)$. Further by application of (res) on $(c/b)\cdot(b/a)\le(c/a)$, one has $b/a\le(c/b)\bsl(c/a)$. By similar method, one has $a\bsl b\le(a\bsl c)/(b\bsl c)$.
        \item[(6)] Since one has $(a\cdot(a\bsl(c/b)))\cdot b\le c$ and $(a\cdot((a\bsl c)/b))\cdot b\le c$, then by (res) one has $a\bsl(c/b)\le(a\bsl c)/b$ and $(a\bsl c)/b\le a\bsl(c/b)$. Consequently, one has $a\bsl(c/b)=(a\bsl c)/b$.
        \item[(7)] By applying (res) twice on $a/c\le a/c$, one has $c\le(a/c)\bsl a$. By a similar method, one has $c\le a/(c\bsl a)$.
        \item[(8)] Since $(a/b)\cdot b\le a$ and $b\cdot(b\bsl a)\le a$, by (res) one has $a\bsl b\le b/a$ and $b/a\le a\bsl b$, Consequently, one has $a\bsl b=b/a$.
    \end{proof}

    \begin{Definition}\label{algebra1}
    An algebra $\msf{F}=(A,\land,\vee,\cdot,\bsl,\neg)$ is called a commutative residuated lattice-ordered semigroup with involutive negation (denoted by $Inc\mc{RL^-}$) if and only if:
        \begin{itemize}
            \item $(A,\land,\vee,\cdot,\bsl)$ is a $c\mc{RL^-}$;
            \item $\neg$ is an operation on $A$ satisfying the following properties:
                \begin{itemize}
                   \item[$\mrm{(dn)}$]  $\neg\neg a=a$;
                   \item[$\mrm{(in)}$] $ a\bsl\neg b=b\bsl \neg a$.
                \end{itemize}
        \end{itemize}
    \end{Definition}
     $\msf{F}=(A,\land,\vee,\cdot,\bsl,\neg)$ is a commutative residuated lattice-ordered semigroup with  quasi-involutive  negation (denoted by $qInc\mc{RL^-}$) if it satisfies the condition (dn1) $a\leq \neg \neg a$ instead of (dn). In definition \ref{algebra1}, (in) can be replaced by the following equivalent condition:
        \begin{itemize}
            \item [$\quad$] 
                \begin{center}
                $\mrm{(smn)}$ If $a\cdot b\le\neg c$, then $c\cdot b\le\neg a$ where $b$ can be null.
                \end{center}
        \end{itemize}
    We show (in) implies (smn). Suppose that $a\cdot b \le \neg c$. Then by (res), one obtains $b\leq a\bsl \neg c$. Hence $b\leq c \bsl \neg a$. Therefore $b\cdot c\le \neg a$. Similarly, one can prove that if $c\cdot b\le\neg a$, then $a\cdot b\le\neg c$. Conversely, we show (smn) implies (in).  Since $a \bsl \neg b \le a \bsl \neg b$, by (res), $a \cdot a\bsl \neg b \le \neg b $. Then by (smn), one obtains $(a\bsl \neg b) \cdot (\neg\neg b)\le \neg a$. Thus by (dn), one gets $a \bsl \neg b \cdot b \le \neg a$.
    Therefore by (res), $a\bsl \neg b\le b\bsl \neg a$. The other inequation can be proved similarly. Obviously $c\mc{RL^-}$ enriched with a negation satisfying (smn) is indeed a $qInc\mc{RL^-}$. A  $c\mc{RL^-}$ with a negation satisfies the following (mn) instead of (in) and (dn) is called minimal $c\mc{RL^-}$, denoted by $mnc\mc{RL^-}$. 
    \begin{itemize}
        \item[$\quad$] 
            \begin{center}
            $\mrm{(mn)}$ If $a\le\neg c$, then $c\le\neg a$.
            \end{center}
        \end{itemize}
    Clearly (smn) implies (mn). Moreover by (mn), one obtains $a\leq \neg \neg a$. A $mnc\mc{RL^-}$ enriched with $\neg\neg a\leq a$ is a De morgan $c\mc{RL^-}$. One denoted it by $dnc\mc{RL^-}$. $c\mc{RL}$ enriched with an involutive negation ($Inc\mc{RL}$) is defined naturally.

    \begin{Proposition}
    $\neg a=a\bsl0$ and $a\cdot \neg a= 0$ in $Inc\mc{RL}$ where $0=\neg1$.
    \end{Proposition}

    \begin{proof}
    For $Inc\mc{RL}$ contains unit, then $a\cdot 1=1\cdot a=a$. Since one has $\neg a\le\neg a$, by the unit one has $1\cdot\neg a\le\neg a$. By (smn), one has $a\cdot\neg a\le\neg1$. Then by (res), one has $\neg a\le a\bsl0$. For another direction, by applying (res) on $a\bsl0\le a\bsl0$, one has $a\cdot a\bsl0\le0$. Further by applying (smn), one has $a\cdot a\bsl0\le\neg1$ and $a\bsl0\le\neg a$. Therefore one has $\neg a=a\bsl0$ where $0=\neg1$. Since $a\cdot a\bsl 0 =0$, whence $a\cdot \neg a =0$.
    \end{proof}

    \begin{Example}\label{example1}
    The lattice in Figure \ref{figlattice} is an example of $Inc\mc{RL^-}$.
        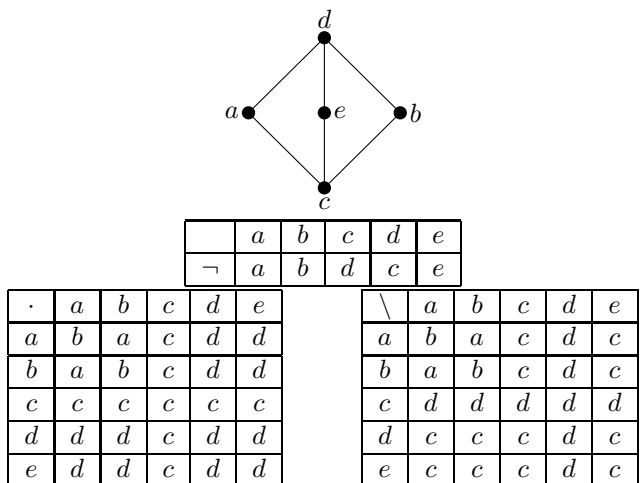
\begin{figure}[htbp]
            \centering
                \begin{tikzpicture}
                    \fill(0,0) circle(2.5pt);
                    \fill(0,-2) circle(2.5pt);
                    \fill(-1,-1) circle(2.5pt);
                    \fill(1,-1) circle(2.5pt);
				\fill(0,-1) circle(2.5pt);
                    \draw (0,0) -- (-1,-1);
                    \draw (-1,-1) -- (0,-2);
                    \draw (0,-2) -- (1,-1);
                    \draw (1,-1) -- (0,0);
				\draw (0,-1) -- (0,0);
				\draw (0,-1) -- (0,-2);
                    \node[above] at (0,0) {$d$};
                    \node[left] at (-1,-1) {$a$};
                    \node[below] at (0,-2) {$c$};
                    \node[right] at (1,-1) {$b$};
				\node[right] at (0,-1) {$e$};
                \end{tikzpicture}
                
                \begin{tabular}{|c|c|c|c|c|c|}
                    \hline
                    $ $&$a$&$b$&$c$&$d$&$e$\\
                    \hline
                    $\neg$&$a$&$b$&$d$&$c$&$e$\\
                    \hline
                \end{tabular}
                
                \begin{tabular}{|c|c|c|c|c|c|}
                    \hline
                    $\cdot$&$a$&$b$&$c$&$d$&$e$\\
                    \hline
                    $a$&$b$&$a$&$c$&$d$&$d$\\
                    \hline
                    $b$&$a$&$b$&$c$&$d$&$d$\\
                    \hline
                    $c$&$c$&$c$&$c$&$c$&$c$\\
                    \hline
                    $d$&$d$&$d$&$c$&$d$&$d$\\
                    \hline
				$e$&$d$&$d$&$c$&$d$&$d$\\
                    \hline
                \end{tabular}
                $\quad\quad\ $
                \begin{tabular}{|c|c|c|c|c|c|}
                    \hline
                    $\bsl$&$a$&$b$&$c$&$d$&$e$\\
                    \hline
                    $a$&$b$&$a$&$c$&$d$&$c$\\
                    \hline
                    $b$&$a$&$b$&$c$&$d$&$c$\\
                    \hline
                    $c$&$d$&$d$&$d$&$d$&$d$\\
                    \hline
                    $d$&$c$&$c$&$c$&$d$&$c$\\
                    \hline
				$e$&$c$&$c$&$c$&$d$&$c$\\
                    \hline
                \end{tabular} 
            \caption{\centering{An example of $Inc\mc{RL^-}$}}
        \label{figlattice}
        \end{figure}
        
    \end{Example}
    Evidently, Example \ref{example1} is an instance of $Inc\mc{RL^-}$ but not even a subalgebra of an instance of $Inc\mc{RL}$. Since there is no special element $y\in\{a,b,c,d,e\}$ satisfying that for any $x\in\{a,b,c,d,e\}$, $x\cdot \neg x= y$. Therefore this example shows that $Inc\mc{RL^-}$ is essentially different with $Inc\mc{RL}$.

    \begin{Proposition}\label{propositionnegation}
    The following properties hold in $Inc\mc{RL^-}$ for all $a,b,c\in A$:
        \begin{itemize}
            \item[$\mrm{(mn)}$] If $a\le\neg c$, then $c\le\neg a$;
            \item[$\mrm{(dn1)}$] $a\le\neg\neg a$;
            \item[$\mrm{(smn')}$] If $a\cdot b\le c$, then $\neg c\cdot b\leq\neg a$;
            \item[$\mrm{(ctr)}$] If $a\le b$, then $\neg b\le\neg a$.
        \end{itemize}
    \end{Proposition}

    \begin{proof}
    \ 
        \begin{itemize}
            \item[$\mrm{(mn)}$] Let $b$ be null, one has this condition by (smn)'s simple substitution.
            \item[$\mrm{(dn1)}$] By applying (mn) on $\neg a\le\neg a$, one has $a\le\neg\neg a$.
            \item[$\mrm{(smn')}$] Assume $a\cdot b\le c$, one has $a\cdot b\le\neg\neg c$ by (dn). Further by applying (smn), one has $a\cdot b\le c$ implies $\neg c\cdot b\leq\neg a$.
            \item[$\mrm{(ctr)}$] Assume $a\le b$, by (dn1) one has $b\le\neg\neg b$. Therefore one has $a\le\neg\neg b$. By (mn) one has $\neg b\le\neg a$. 
        \end{itemize}
    \end{proof}

    A negation satisfying only the contraposition law (ctr) is called subminimal negation. Clearly (mn) implies (ctr). Let $a\leq b$. By (dn1), $b\leq \neg\neg b$. Thus $a\leq \neg\neg b$. Then by (mn), one obtains $\neg b\leq \neg a$. So a minimal negation is indeed a subminimal negation. As we have discussed above, a De morgan negation is a minimal negation satisfying $\neg\neg a \leq a$, while an involutive negation is a De morgan negation satisfying $a \bsl \neg b= b \bsl \neg a$. Therefore, following Dunn's Kite of Negations (cf. \cite{dunn1999comparative}), we present a string of negations over $c\mc{RL^-}$ in Figure \ref{figstring}. Note that L (Lambek calculus) with De morgan negation was introduced by Buszkowski \cite{buszkowski1995categorial} in order to describe the negative information of categorial grammar. However, the cut-free sequent calculus and the decidability problem of the corresponding logic still remain open.

    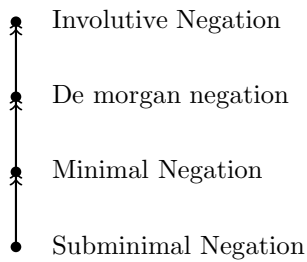
\begin{figure}[htbp]
        \centering
            \begin{tikzpicture}
                \fill(0,3) circle(2pt);
                \fill(0,2) circle(2pt);
                \fill(0,1) circle(2pt);
                \fill(0,0) circle(2pt);
                \draw [->>,thick] (0,0) -- (0,1);
                \draw [->>,thick] (0,1) -- (0,2);
                \draw [->>,thick] (0,2) -- (0,3);
                \node[right]at (0,0) {$\quad$Subminimal Negation};
                \node[right] at (0,1) {$\quad$Minimal Negation};
                \node[right] at (0,2) {$\quad$De morgan negation};
                \node[right] at (0,3) {$\quad$Involutive Negation};
            \end{tikzpicture}
        \caption{\centering{The string of negations}}
        \label{figstring}
    \end{figure}
 
\section{Logic}\label{section3}
    
    In this section, we are going to present InFL$^-_e$ (denoted by $\msf{G}$) and qInFL$^-_e$ which are the logics of $Inc\mc{RL^-}$ and $qInc\mc{RL^-}$ respectively. Simple sequent calculi for these algebras will be established. Soundness and Completeness results will be proved as well. Finally, we will show that InFL$^-_e$ can be embedded into qInFL$^-_e$.

    \begin{Definition}
    The set of formulas (terms) $\mc{F}$ for InFL$^-_e$ is defined inductively as follows:
    \[\mc{F}\ni \al::= p\mid\al\cdot\be \mid \al\bsl\be \mid\al\land\be \mid \al\vee\be \mid \neg\al\]
    where $p\in \mbf{Var}$. 
    \end{Definition}

    \begin{Definition}{\em
    The simple Gentzen-style sequent calculus $\msf{G}$ for commutative residuated lattice-ordered semigroup with involutive negation consists of the following axioms and rules:
        \begin{itemize}
            \item[$(1)$] Axioms:
                \[
                (\mrm{Id}) \al\Ra\al
                \quad
                (\mrm{DN2}) \neg\neg \alpha\Ra\alpha
                \]
            \item[$(2)$] Logical rules:
                \[
                \frac{\al\cdot\be \Rightarrow \gamma}{\be\Rightarrow \al\bsl\gamma}(\mrm{RES}\bsl)
                \quad
                \frac{\be\Rightarrow \al\bsl\gamma}{\al\cdot\be \Rightarrow \gamma}(\mrm{RES^-}\bsl)
                \]
                \[
                \frac{\al\Ra\be}{\al\land\gamma\Ra\be}{({\land}\mrm{L})}
                \quad
                \frac{\al\Ra\be \quad \al\Ra\gamma}{\al\Ra\be\land\gamma}{({\land}\mrm{R})}
                \]
                \[
                \frac{\al\Ra\be \quad \gamma\Ra\be}{\al\vee\gamma\Ra\be}{({\vee}\mrm{L})}
                \quad
                \frac{\al\Ra\be}{\al\Ra\be\vee\gamma}{({\vee}\mrm{R})}
                \]
                \[
                \frac{\al\cdot\be\Ra\neg\gamma}{\gamma\cdot\be\Ra\neg\al}{(\neg)}
                \]
            \item[(3)] Cut rule:
                \[
                \frac{\al\Ra \be \quad \be\Ra \gamma}{\al\Ra\gamma}{(\mrm{Cut})}
                \]
        \end{itemize}
    }
    \end{Definition}

    A sequent $\al\Ra \be$ is provable in $\msf{G}$, notation $\vdash_{\msf{G}}\al\Ra \be$, if there is a derivation of $\al\Ra \be$ in $\msf{G}$. We write $\vdash_{\msf{G}}\al\Leftrightarrow \be$ if $\vdash_{\msf{G}}\al\Ra \be$ and $\vdash_{\msf{G}}\be\Ra \al$. Since $\beta$ in the rule $(\neg)$ can be null i.e. rule $(\neg)$ can be the following special form:
        \[
        \frac{\alpha\Ra\neg\gamma}{\gamma\Ra\neg\alpha}{(\mrm{MN})}
        \]
    In addition, if one omits (DN2): $\neg\neg\al\Ra\al$, then we will get a simple sequent calculus for qInFL$^-_e$.

    \begin{Property}\label{property4}
    The following rules and axioms are admissible in $\msf{G}$:
        \begin{itemize}
            \item[$(\mrm{DN1})$] $\vdash_{\msf{G}}\alpha\Ra\neg\neg\alpha$;
            \item[$(\neg')$] If $\vdash_{\msf{G}}\alpha\cdot\beta\Ra\gamma$, then $\vdash_{\msf{G}}\neg\gamma\cdot\beta\Ra\neg\alpha$;
            \item[$\mrm{(DN\land)}$] $\vdash_{\msf{G}}\neg\neg(\neg\neg\al\land\neg\neg\be)\LR\neg\neg\al\land\neg\neg\be$;
            \item[$\mrm{(Mon)}$] If $\vdash_{\msf{G}}\al_1\Ra\be_1$ and $\al_2\Ra\be_2$, then $\vdash_{\msf{G}}\al_1\cdot\al_2\Ra\be_1\cdot\be_2$;
            \item[$\mrm{(DN\bsl)}$] $\vdash_{\msf{G}}\neg\neg(\neg\neg\al\bsl\neg\neg\be)\LR\neg\neg\al\bsl\neg\neg\be$.
        \end{itemize}
    \end{Property}

    \begin{proof}
    We only give the proofs of (Mon) and (DN$\bsl$).
        \begin{itemize}
            \item[(Mon)]:
                \begin{prooftree}
                    \AxiomC{$\al_1\Ra\be_1$}
                    \AxiomC{$\be_1\cdot\al_2\Ra\be_1\cdot\al_2$}
                    \RightLabel{$\mrm{(RES\bsl)}$}
                    \UnaryInfC{$\be_1\Ra\al_2\bsl(\be_1\cdot\al_2)$}
                    \RightLabel{(Cut)}
                    \BinaryInfC{$\al_1\Ra\al_2\bsl(\be_1\cdot\al_2)$}
                    \RightLabel{$\mrm{(RES^-\bsl)}$}
                    \UnaryInfC{$\al_1\cdot\al_2\Ra\be_1\cdot\al_2$}
                    \AxiomC{$\al_2\Ra\be_2$}
                    \AxiomC{$\be_1\cdot\be_2\Ra\be_1\cdot\be_2$}
                    \RightLabel{$\mrm{(RES\bsl)}$}
                    \UnaryInfC{$\be_2\Ra\be_1\bsl(\be_1\cdot\be_2)$}
                    \RightLabel{(Cut)}
                    \BinaryInfC{$\al_2\Ra\be_1\bsl(\be_1\cdot\be_2)$}
                    \RightLabel{$\mrm{(RES^-\bsl)}$}
                    \UnaryInfC{$\be_1\cdot\al_2\Ra\be_1\cdot\be_2$}
                    \RightLabel{(Cut)}
                    \BinaryInfC{$\al_1\cdot\al_2\Ra(\be_1\cdot\be_2)$}
                \end{prooftree}
            \item[$\mrm{(DN\bsl)}$]
            From right to left, it is a substitution of (DN1). From left to right:
            \begin{prooftree}
                \AxiomC{$\neg\neg\alpha\bsl\neg\neg\beta\Ra\neg\neg\alpha\bsl\neg\neg\beta$}
                \RightLabel{ $(\mrm{RES^-\bsl})$}
                \UnaryInfC{$\neg\neg\alpha\cdot(\neg\neg\alpha\bsl\neg\neg\beta) \Rightarrow\neg\neg\beta$}
                \RightLabel{ $(\neg')\times2$}
                \UnaryInfC{$\neg\neg\alpha\cdot\neg\neg(\neg\neg\alpha\bsl\neg\neg\beta) \Rightarrow\neg^4\beta$}
                \AxiomC{$\be\Ra\neg\neg\be$}
                \RightLabel{($\neg')\times2$}
                \UnaryInfC{$\neg^4\beta \Rightarrow\neg\neg\beta$}
                \RightLabel{ $(\mrm{Cut})$}
                \BinaryInfC{$\neg\neg\alpha\cdot\neg\neg(\neg\neg\alpha\bsl\neg\neg\beta) \Rightarrow\neg\neg\beta$}
                \RightLabel{ $(\mrm{RES\bsl})$}
                \UnaryInfC{$\neg\neg(\neg\neg\alpha\bsl\neg\neg\beta) \Rightarrow \neg\neg\alpha\bsl\neg\neg\beta$}
            \end{prooftree}
        \end{itemize}
    \end{proof}    

    Let ($A,\wedge,\vee,\cdot,\bsl,\neg$) be an $Inc\mc{RL^-}$. Let $\mbf{Var}$ be a set of all atomic formulas $p,q,r,\ldots$ and $\mu$ be a mapping from $\mbf{Var}$ to $A$. We write $\mu(p)$ as the corresponding element of $p$ in $A$. This mapping can be extended to the set of formulas $\mathcal{F}$ naturally by the following way: $\mu(\neg \alpha)=\neg(\mu(\alpha))$ and $\mu(\alpha\star\beta)=\mu(\alpha)\star\mu(\beta)$ for any $\alpha,\beta\in \mathcal{F}$ and $\star\in\{\wedge,\vee,\cdot,\bsl\}$. An $Inc\mc{RL^-}$ model is a pair ($\mathbf{A}, \mu)$ such that $\mathbf{A}$ is an $Inc\mc{RL^-}$ and $\mu$ is a mapping from $\mbf{Var}$ to $A$. A formula is true for ($\mathbf{A}, \mu)$ if $\mu(\alpha)=1$ in $\mathbf{A}$. A formula is valid for the class of $Inc\mc{RL^-}$s if $\mu(\alpha)=1$ for any $Inc\mc{RL^-}$ model ($\mathbf{A}, \mu)$. We write  $\vDash_{Inc\mc{RL^-}} \alpha$ if $\alpha$ is valid for the class of $Inc\mc{RL^-}$s.

    \begin{Definition}
    A simple sequent calculus $\msf{G}$ is called sound with respect to a class of algebras $\mc{K}$, if for any sequent $\alpha \Rightarrow \beta$, $\vdash_{\msf{G}}\alpha\Ra\beta$ implies $\vDash_\mc{K}\alpha\Ra\beta$. A simple sequent calculus $\msf{G}$ is called complete with respect to $\mc{K}$, if for any sequent $\alpha \Rightarrow \beta$, $\vDash_\mc{K}\alpha\Ra\beta$ implies $\vdash_{\msf{G}}\alpha\Ra\beta$.
    \end{Definition}

    \begin{Theorem}[Soundness and Completeness]\label{complete}
    $\msf{G}$ is sound and complete with respect to $Inc\mc{RL^-}$s.
    \end{Theorem}
    
    \begin{proof}
    The soundness is followed by the properties and propositions we have presented in the previous paragraphs. Rules related to lattice, fusion, residual, and negation are valid by Property \ref{propertylattice}, Property \ref{propertyfusion} and \ref{propertyresidual} and Proposition \ref{propositionnegation} respectively. For completeness, it suffices to show that for any sequent $\alpha \Ra \beta$, if $\not\vdash_{\msf{G}}\alpha\Ra \beta $, then $\not\vDash_{Inc\mc{RL^-}} \alpha \Ra \beta$. It can be proved by standard construction.
    Let $\den{\alpha}=\{\beta\mid\ \vdash_{\msf{G}}\alpha\Leftrightarrow\beta\}$. Let $A$ be the set of all $\den{\alpha}$. Define $ \wedge',\vee',\cdot',\bsl',\neg'$ on $A$ as follows:
    \[\den{\alpha_1}\wedge'\den{\alpha_2}=\den{\alpha_1\wedge\alpha_2}\quad\den{\alpha_1}\vee'\den{\alpha_2}=\den{\alpha_1\vee\alpha_2}
    \]
    \[
    \den{\alpha_1}\cdot'\den{\alpha_2}=\den{\alpha_1\cdot\alpha_2}\quad \den{\alpha_1}\bsl'\den{\alpha_2}=\den{\alpha_1\bsl\alpha_2}
    \]
    \[
    \neg'\den{\alpha}=\den{\neg\alpha}
    \]
    Clearly by Definition \ref{algebra1}, $\mbf{A}=(A,\wedge',\vee',\cdot',\bsl',\neg'$) is an $Inc\mc{RL^-}$. The lattice order is defined as $\den{\alpha_1}\leq' \den{\alpha_2}$ iff $\den{\alpha_1}\wedge'\den{\alpha_2}=\den{\alpha_1}$. Thus $\den{\alpha_1}\leq' \den{\alpha_2}$ iff $\vdash_{G}\alpha_1\Ra\alpha_2$. Define an assignment $\sigma:\mbf{Var}\longrightarrow A$ such that $\sigma(p)=\den{p}$. By induction on the complexity of the formula, one shows that $\widehat{\sigma}(\alpha)=\den{\alpha}$ for some formulas $\alpha$. Suppose that $\vDash_{Inc\mc{RL^-}} \alpha \Ra \beta$. Then $\widehat{\sigma}(\alpha)\leq \widehat{\sigma}(\beta)$ in $\mbf{A}$. Hence $\vdash_{G} \alpha \Ra \beta$, which yields a contradiction. This completes the proof.
    \end{proof}

    We define a new system that has been mentioned previously: the sequent calculus for commutative residuated lattice-ordered semigroup with quasi-involutive negation. This calculus is just $\msf{G}$ omitting $\neg\neg\al\Ra\al$, denoted by $\msf{qG}$. It is the intuitionistic counterpart of $\msf{G}$. Next, we shall present the Kolmogorov translation and show the embeddability result between $\msf{G}$ and $\msf{qG}$. Note that Galatos and Ono \cite{galatos2006glivenko} prove that for every involutive extension of FL, there exists a minimum extension of FL that contains the former via a double negation interpretation.

    \begin{Corollary}
    InFL$^-_e$ is sound and complete with respect to $Inc\mc{RL^-}$.
    \end{Corollary}

    \begin{Definition}[Kolmogorov Translation]\label{Kolmogorov}
    The Kolmogorov translation is a function $ko:\mc{F}\ra\mc{F}$ defined inductively as follows:
        \begin{itemize}
            \item[(1)] $ko(\al)=\neg\neg\al$ if $\al$ is atomic;
            \item[(2)] $ko(\al\land\be)=\neg\neg(ko(\al)\land ko(\be))$;
            \item[(3)] $ko(\al\vee\be)=\neg\neg(ko(\al)\vee ko(\be))$;
            \item[(4)] $ko(\neg\al)=\neg ko(\al)$;
            \item[(5)] $ko(\al\cdot\be)=\neg\neg(ko(\al)\cdot ko(\be))$;
            \item[(6)] $ko(\al\bsl\be)=\neg\neg(ko(\al)\bsl ko(\be))$.
        \end{itemize}
    \end{Definition}

    \begin{Lemma}\label{lemma1}
    For every formula $\al$, $\vdash_{\msf{qG}}ko(\neg\neg\al)\Ra ko(\al)$.
    \end{Lemma}

    \begin{proof}
    one has $ko(\neg\neg\al)=\neg\neg ko(\al)$ by Definition \ref{Kolmogorov}. Let $ko(\al)=\neg\be$ for some formula $\be$. By (DN1) and (MN) one has $\vdash_{\msf{G}}\neg\neg\neg\be\Ra\neg\be$. Therefore, one has $\vdash_{\msf{G}}ko(\neg\neg\al)\Ra ko(\al)$.
    \end{proof}

    \begin{Lemma}\label{embeddability}
    $\vdash_{\msf{qG}}ko(\al)\Ra ko(\be)$ iff $\vdash_{\msf{G}}\al\Ra\be$ where $ko()$ is the Kolmogorov translation.
    \end{Lemma}

    \begin{proof}
    From left to right: Assume $\vdash_{\msf{qG}}ko(\al)\Ra ko(\be)$, since $\msf{G}$ is an extension of $\msf{qG}$, one has $\vdash_{\msf{G}}ko(\al)\Ra ko(\be)$. Since double negation holds in $\msf{G}$, one has $\vdash_{\msf{G}}\al\Ra\be$.

    From right to left: Assume $\vdash_{\msf{G}}\al\Ra\be$, we prove this direction by induction on the length of derivation n. Suppose n=0, then $\al\Ra\be$ is an instance of (Id). Then one has $\al=\be$. Suppose $\al=\neg\neg\al'$, one has $\vdash_{\msf{qG}}ko(\neg\neg\al')\Ra ko(\al')$ by Lemma \ref{lemma1}. Thus the claim $\vdash_{\msf{qG}}ko(\al)\Ra ko(\be)$ holds for n=0. Suppose n>0, assume sequent $\al\Ra\be$ is obtained by an arbitrary rule (R), one has the following cases:
        \begin{itemize}
            \item[(1)] (R) is (Cut). Let the premises be $\vdash_{\msf{G}}\al\Ra\ga$ and $\vdash_{\msf{G}}\ga\Ra\be$. By induction hypothesis, one has $\vdash_{\msf{qG}}ko(\al)\Ra ko(\ga)$ and $\vdash_{\msf{qG}}ko(\ga)\Ra ko(\be)$. By (Cut) one has $\vdash_{\msf{qG}}ko(\al)\Ra ko(\be)$.
            
            \item[(2)] (R) is (RES$\bsl$). Let the premise be $\vdash_{\msf{G}}\al\cdot\be\Ra\ga$. By induction hypothesis, one has $\vdash_{\msf{qG}}ko(\al\cdot\be)\Ra ko(\ga)$. By (DN1), one has $\vdash_{\msf{qG}}ko(\al)\cdot ko(\be)\Ra\neg\neg(ko(\al)\cdot ko(\be))$. By applying (Cut), one has $\vdash_{\msf{qG}}ko(\al)\cdot ko(\be)\Ra ko(\ga)$. Next by (RES$\bsl$) one has $\vdash_{\msf{qG}}ko(\be)\Ra ko(\al)\bsl ko(\ga)$. Since one has $\vdash_{\msf{qG}}ko(\al)\bsl ko(\ga)\Ra\neg\neg(ko(\al)\bsl ko(\ga))$ by (DN1), one gets $\vdash_{\msf{qG}}ko(\be)\Ra\neg\neg(ko(\al)\bsl ko(\ga))$ by (Cut). Therefore, one has $\vdash_{\msf{qG}}ko(\be)\Ra ko(\al\bsl\ga)$.
            
            \item[(3)] (R) is (RES$^-\bsl$). Let the premise be $\vdash_{\msf{G}}\be\Ra\al\bsl\ga$. By induction hypothesis, one has $\vdash_{\msf{qG}}ko(\be)\Ra ko(\al\bsl\ga)$. By applying (DN$\bsl$) in Property \ref{property4}, one has $\vdash_{\msf{qG}}\neg\neg(ko(\al)\bsl ko(\ga))\LR ko(\al)\bsl ko(\ga)$. By applying (RES$^-\bsl$) and (Cut), one has $\vdash_{\msf{qG}}ko(\al)\cdot ko(\be)\Ra ko(\ga)$. By applying ($\neg'$), one has $\vdash_{\msf{qG}}\neg\neg(ko(\al)\cdot ko(\be))\Ra\neg\neg ko(\ga)$. Since $\neg\neg ko(\ga)=ko(\neg\neg\ga)$, one has $\vdash_{\msf{qG}}\neg\neg(ko(\al)\cdot ko(\be))\Ra ko(\ga)$ by applying (Cut) and Lemma \ref{lemma1}. Therefore, one has $\vdash_{\msf{qG}}ko(\al\cdot\be)\Ra ko(\ga)$.

            \item[(4)] (R) is ($\land$L). Let the premise be $\vdash_{\msf{G}}\al\Ra\be$. By induction hypothesis, one has $\vdash_{\msf{qG}}ko(\al)\Ra ko(\be)$. By applying ($\land$L), one has $\vdash_{\msf{qG}}ko(\al)\land ko(\ga)\Ra ko(\be)$. By applying (DN$\land$) in Property \ref{property4}, one has $\vdash_{\msf{qG}}\neg\neg(ko(\al)\land ko(\ga))\Ra ko(\be)$. Therefore, one has $\vdash_{\msf{qG}}ko(\al\land\ga)\Ra ko(\be)$.

            \item[(5)] (R) is ($\land$R). Let the premises be $\vdash_{\msf{G}}\al\Ra\be$ and $\vdash_{\msf{G}}\al\Ra\ga$. By induction hypothesis, one has $\vdash_{\msf{qG}}ko(\al)\Ra ko(\be)$ and $\vdash_{\msf{qG}}ko(\al)\Ra ko(\ga)$. By ($\land$R) and (DN1) one has $\vdash_{\msf{qG}}ko(\al)\Ra\neg\neg(ko(\be)\land ko(\ga))$. Therefore, one has $\vdash_{\msf{qG}}ko(\al)\Ra ko(\be\land\ga))$.

            \item[(6)] (R) is ($\vee$L). Let the premises be $\vdash_{\msf{G}}\al\Ra\be$ and $\vdash_{\msf{G}}\ga\Ra\be$. By induction hypothesis, one has $\vdash_{\msf{qG}}ko(\al)\Ra ko(\be)$ and $\vdash_{\msf{qG}}ko(\ga)\Ra ko(\be)$. By applying ($\vee$L), one has $\vdash_{\msf{qG}}ko(\al)\vee ko(\ga)\Ra ko(\be)$. By applying ($\neg'$), one has $\vdash_{\msf{qG}}\neg\neg(ko(\al)\vee ko(\ga))\Ra\neg\neg ko(\be)$. Since $\neg\neg ko(\be)=ko(\neg\neg\be)$, one has $\vdash_{\msf{qG}}\neg\neg(ko(\al)\vee ko(\ga))\Ra ko(\be)$ by applying (Cut) and Lemma \ref{lemma1}. Therefore, one has $\vdash_{\msf{qG}}ko(\al\vee\ga)\Ra ko(\be)$.

            \item[(7)] (R) is ($\vee$R). Let the premise be $\vdash_{\msf{G}}\al\Ra\be$. By induction hypothesis, one has $\vdash_{\msf{qG}}ko(\al)\Ra ko(\be)$. By ($\vee$R) and (DN1) one has $\vdash_{\msf{qG}}ko(\al)\Ra\neg\neg(ko(\be)\vee ko(\ga))$. Therefore, one has $\vdash_{\msf{qG}}ko(\al)\Ra ko(\be\vee\ga)$.

            \item[(8)] (R) is ($\neg$). Let the premise be $\vdash_{\msf{G}}\al\cdot\be\Ra\neg\ga$. By induction hypothesis, one has $\vdash_{\msf{qG}}ko(\al\cdot\be)\Ra ko(\neg\ga)$. By applying (Cut), one has $\vdash_{\msf{qG}}ko(\al)\cdot ko(\be)\Ra\neg ko(\ga)$. By applying ($\neg$) one has $\vdash_{\msf{qG}}ko(\ga)\cdot ko(\be)\Ra\neg ko(\al)$. Further by (MN) and ($\neg'$), one has $\vdash_{\msf{qG}}\neg\neg(ko(\ga)\cdot ko(\be))\Ra\neg ko(\al)$. Therefore, one has $\vdash_{\msf{qG}}ko(\ga\cdot\be)\Ra ko(\neg\al)$.
        \end{itemize}
    \end{proof}

    \begin{Theorem}[Embeddability]\label{thmembeddability}
    $\msf{G}$ can be embedded into $\msf{qG}$.
    \end{Theorem}

    \begin{proof}
    It immediately follows from Lemma \ref{embeddability}.
    \end{proof}

    \begin{Corollary}
    $Inc\mc{RL^-}$ can be embedded into $qInc\mc{RL^-}$.
    \end{Corollary}

    \begin{proof}
    It immediately follows from Theorem \ref{thmembeddability} and soundness result of corresponding algebras.
    \end{proof}
    
\section{Decidability}\label{section4}

    In this section, we shall consider a conservative extension of $Inc\mc{RL^-}$. By showing the cut-elimination result of such extension's sequent calculus, we will prove the decidability of InFL$^-_e$.

    \begin{Definition}\label{algebra2}
    An algebra $\ddot{\msf{F}}=(A,\wedge,\vee,\cdot,\bsl,\ast,\rightarrow,\bot, \top)$ is a bounded commutative bi-residuated lattice-ordered semigroup (denoted by $cb\mc{RL^-}$) such that $(A,\wedge,\vee,\cdot,\bsl)$ is a $c\mc{RL^-}$, and $(A,\land,\vee,\ast,\ra)$ is a commutative residuated lattice-ordered groupoid. $\bot,\top \in A$ are the least and greatest elements in $A$. Two fusions $\ast,\cdot$ are operations on $A$ satisfying the following condition for all $a,b,c\in A$:
    \[
    \mbox{If} (a\cdot b)\ast c =\bot,\ \mbox{then} (a\cdot c)\ast b =\bot.
    \]
    One defines $\neg a=a\ra\bot$, then one has $a\ast\neg a\le\bot$. Moreover a $cb\mc{RL^-}$ is indeed an $Inc\mc{RL^-}$ by the following proposition.
    \end{Definition}

    \begin{Proposition}
    The following properties still hold for $cb\mc{RL^-}$ for all $a,b,c\in A$:
        \begin{itemize}
            \item[$\mrm{(mon)}$] If $a_1\le b_1$ and $a_2\le b_2$, then $a_1\ast b_1\le a_2\ast b_2$;
            \item[$\mrm{(smn)}$] If $a\cdot b\le\neg c$, then $c\cdot b\le\neg a$ where $b$ can be null;
            \item[$\mrm{(dn1)}$] $a\le\neg\neg a$.
        \end{itemize}
    \end{Proposition}

    \begin{proof}
    \
        \begin{itemize}
            \item[$\mrm{(mon)}$] The proof is quite similar to the proof of Property \ref{propertyfusion} (3).
            \item[$\mrm{(smn)}$] Assume $a\cdot b\le\neg c$, then one has $(a\cdot b)\ast c\le\neg c\ast c$ by (mon) together with $c\le c$. By Definition \ref{algebra2}, one has $c\ast\neg c\le \bot$. Therefore one has $(a\cdot b)\ast c\le \bot$. Next by Definition \ref{algebra2}, one has $(c\cdot b)\star a\le \bot$, whence $c\cdot b\le\neg a$. Therefore, $a\cdot b\le\neg c$ implies $c\cdot b\le\neg a$.
            \item[$\mrm{(dn1)}$] The proof is same as the proof of Proposition \ref{propositionnegation} (dn1).
        \end{itemize}
    \end{proof}

    \begin{Theorem}\label{expand}
    Every $qInc\mc{RL^-}$ can be extended to a $cb\mc{RL^-}$.
    \end{Theorem}
    
    \begin{proof}
    Given a $qInc\mc{RL^-}$, from which we construct an algebra $M=(A,\wedge,\vee,\cdot,\bsl,\neg)$. Let $\top=\bigvee A$ be the greatest element of $A$, $\bot=\bigwedge A$ be the least element of $A$. $\ast,\rightarrow$ are operations on $A$ defined as follow:
        \[a \ast b=\left\{
        \begin{aligned}
        \bot \quad \mrm{if} \quad a\leq \neg b \\
        \top \quad \mrm{otherwise}
        \end{aligned}
        \right.
        \quad
        a \rightarrow b=\left\{
        \begin{aligned}
        \neg a \quad \mrm{if} \quad b\neq\top \\
        \top \quad \mrm{otherwise}
        \end{aligned}
        \right.
        \]
    Then the following properties hold in $M$ for any $a,b \in A$:
    \begin{itemize}
        \item[(1)] $a\ast b\leq c$ iff $b \leq a\rightarrow c$;
        \item[(2)] If $(a\cdot b)\ast c = \bot$, then $ (a\cdot c)\ast b = \bot$;
        \item[(3)] $a \ast b = b \ast a$;
        \item[(4)] $\neg a = a\rightarrow\bot$;
        \item[(5)] $a\ast\neg a=\bot$.
    \end{itemize}
    
        \begin{itemize}
            \item[(1)] Assume $a\ast b\leq c$, we consider two cases. (i) $c=\top$, then $a\ra c=\top$ by the construction above. Thus, $b\leq\top=a\ra c$. (ii) $c\neq\top$, then $c<\top$ and $a\ra c =\neg a$ by the construction. Since $a\ast b\leq c$, then $a\ast b =\bot$. Then by the construction, one has $a \leq\neg b$. Hence by Definition \ref{algebra1}, $\neg\neg b \leq \neg a$ and $b\leq\neg\neg b$. Whence $b\leq \neg a$ i.e. $b \leq a \ra c$. Assume $b \leq a \ra c$, we consider two cases. (i) $c=\top$, then $a\ast b\leq c=\top$. (ii) $c\neq\top$, then $a\ra c =\neg a$ by the construction. Hence $b \leq a \ra c=\neg a$, then $\neg\neg a \leq \neg b$ and $a\leq \neg\neg a$ by Definition \ref{algebra1}. Therefore, $a\leq \neg b$, then $a\ast b=\bot$ by the construction. So $a\ast b\leq c$.

            \item[(2)] Assume $(a\cdot b)\ast c = \bot$, then $a\cdot b\leq \neg c$ by the construction, one has (i) $a\cdot\neg\neg c\leq \neg b$ by (smn). From $a\leq a$ and $c\leq\neg\neg c$, we obtain (ii) $a\cdot c\leq a\cdot\neg\neg c$ by $(\mrm{mon})$. Thus one gets $a\cdot c\leq\neg b $ from (i) and (ii) by Transitivity. Therefore one has $(a\cdot c)\ast b =\bot$ by the construction.

            \item[(3)] We consider two cases, (i) $a\ast b=\bot$, then by the construction above $a\leq\neg b$. Hence by (smn), one has $\neg\neg b\leq\neg a$. Whence $b\leq\neg a$ by Definition \ref{algebra1}. Hereafter $b\ast a=\bot$ by the construction. (ii) $a\ast b=\top$, assume $b\ast a=\bot$, then by (i) one has $a\ast b=\bot$ which generates a contradiction. Therefore, $b\ast a=\top$ i.e. $a\ast b=b\ast a$.

            \item[(4)] By the construction above, $a\ra \bot=\neg a$.

            \item[(5)] Since $a\ra\bot\le a\ra\bot$, one has $a\ast\neg a\le\bot$ by (1) and (4). Since $\bot\le a\ast\neg a$, one has $a\ast\neg a=\bot$.
        \end{itemize}
   
    Therefore, the constructed algebra $M$ from the $qInc\mc{RL^-}$ is a $cb\mc{RL^-}$.
    \end{proof}

    \begin{Remark}
    Note that two fusions $\ast,\cdot$ in $cb\mc{RL^-}$ are essentially different. Associativity does not hold for $\ast$. Take $(a\ast b)\ast c$ and $a\ast(b\ast c)$ with $a\le\neg b$ but $b\nleq\neg c$ for example. Then one has $(a\ast b)\ast c=\bot\ast c=\bot$ but $a\ast(b\ast c)=a\ast\top=\top$. Therefore, $(a\ast b)\ast c\neq a\ast(b\ast c)$ for $\ast$, which is consistent with the fact that ($A,\land,\vee,\ast,\ra$) is a commutative residuated lattice-ordered groupoid in Definition \ref{algebra2}.
    \end{Remark}

    In what follows we are going to present the logic for the $cb\mc{RL^-}$ denoted by bFL$^-_e$. The cut-elimination and the decidability result for bFL$^-_e$ will be proved.

    \begin{Definition}
    The set of formulas (terms) $\mathcal{F}$ for bFL$^-_e$ is defined inductively as follows:
    \[\mathcal{F}\ni \alpha::= p\mid  \alpha\cdot\beta \mid \alpha\bsl\beta \mid \alpha\wedge\beta \mid \alpha\vee\beta \mid \alpha\ast\beta \mid \alpha\rightarrow\beta \mid \bot
    \]
    where $p\in \mbf{Var}$. We use the abbreviations $\neg\alpha:=\alpha\rightarrow\bot$ and $\top=\neg \bot$.
    \end{Definition}
    
    \begin{Definition}
    Let $,$ and $;$ be structural counterparts for $\cdot$ and $\ast$ respectively. The set of all formula
    structures $\mathcal{FS}$ is defined inductively as follows:
    \[\mathcal{FS} \ni \Gamma ::= \alpha \mid \Gamma,\Gamma\mid \Gamma;\Gamma
    \]
    A \textit{sequent} is an expression of the form $\Gamma\Ra \alpha$ where $\Gamma$ is a formula structure and $\alpha$ is a formula. A \textit{context} is a formula structure $\Gamma[-]$ with a designated position  $[-]$ which can be filled with a formula structure. In particular, a single position $[-]$ is a context. For instance $\Gamma[\Delta]$ is the formula structure obtained from $\Gamma[-]$ by substituting $\Delta$ for $[-]$. By $f(\Gamma)$ we denote the formula obtained from $\Gamma$ by replacing all structure operations with their corresponding formula connectives.
    \end{Definition}
    
    \begin{Definition}{\em
    The Gentzen-style sequent calculus $\msf{Gb}$ for bFL$^-_e$ consists of the following axiom and rules:
    \begin{itemize}
        \item[$(1)$] Axiom:
        \[
        (\mrm{Id}) \alpha\Ra\alpha
        \]
        \item[$(2)$] Logical rules:
        \[
        \frac{\Gamma[\alpha,\beta]\Ra\gamma}{\Gamma[\alpha\cdot\beta]\Ra\gamma}{({\cdot}\mrm{L})}
        \quad
        \frac{\Gamma_1 \Ra\alpha\quad\Gamma_2\Ra\beta}{\Gamma_1,\Gamma_2\Ra\alpha\cdot\beta}{({\cdot}\mrm{R})}
        \]
        \[
        \frac{\Gamma[\alpha;\beta]\Ra\gamma}{\Gamma[\alpha\ast\beta]\Ra\gamma}{({\ast}\mrm{L})}
        \quad
        \frac{\Gamma_1 \Ra\alpha\quad\Gamma_2\Ra\beta}{\Gamma_1;\Gamma_2\Ra\alpha\ast\beta}{({\ast}\mrm{R})}
        \]
        
        \[
        \frac{ \Delta\Ra\alpha\quad\Gamma[\beta]\Ra \gamma }{\Gamma[\Delta,\alpha\bsl\beta]\Ra \gamma}{(\bsl \mrm{L})}
        \quad
        \frac{\alpha,\Gamma\Ra \beta}{\Gamma\Ra \alpha\bsl\beta}{(\bsl \mrm{R})}
        \]
        
        \[
        \frac{ \Delta\Ra\alpha\quad\Gamma[\beta]\Ra \gamma }{\Gamma[\Delta;\alpha\ra\beta]\Ra \gamma}{(\ra \mrm{L})}
        \quad
        \frac{\alpha;\Gamma\Ra \beta}{\Gamma\Ra \alpha\ra\beta}{(\ra \mrm{R})}
        \]
        
        \[
        \frac{\Gamma[\alpha]\Ra\beta}{\Gamma[\alpha\wedge\gamma]\Ra\beta}{({\wedge}\mrm{L})}
        \quad
        \frac{\Gamma \Ra\alpha\quad\Gamma\Ra\beta}{\Gamma\Ra\alpha\wedge\beta}{({\wedge}\mrm{R})}
        \]
        \[
        \frac{\Gamma[\alpha]\Ra\gamma \quad \Gamma[\beta]\Ra\gamma}{\Gamma[\alpha\vee\beta]\Ra\gamma}{({\vee}\mrm{L})}
        \quad
        \frac{\Gamma\Ra\alpha}{\Gamma\Ra\alpha\vee\beta}{({\vee}\mrm{R})}
        \]
        \[
        \frac{\Delta\Ra\bot}{\Gamma[\Delta]\Ra\alpha}{(\bot)}
        \]
        \item[$(2)$] Cut rule:
        \[
        \frac{\Delta\Ra \alpha\quad \Gamma[\alpha]\Ra \beta}{\Gamma[\Delta]\Ra\beta}{(\mrm{Cut})}
        \]
        \item[$(3)$] Structure rules:
        \[
        \frac{(\Delta_1,\Delta_2);\Delta_3\Ra \bot}{(\Delta_1,\Delta_3);\Delta_2\Ra \bot}{(\mrm{R}-\bot)}
        \]
        \[
        \frac{\Gamma[\Delta_1,\Delta_2]\Ra \beta}{\Gamma[\Delta_2,\Delta_1]\Ra \beta}{(\mrm{Ex})}
        \quad
        \frac{\Gamma[\Delta_1;\Delta_2]\Ra \beta}{\Gamma[\Delta_2;\Delta_1]\Ra \beta}{(\mrm{Ex^;})}
        \]
        \[
        \frac{\Gamma[\Delta_1,(\Delta_2,\Delta_3)]\Ra \beta}{\Gamma[(\Delta_1,\Delta_2),\Delta_3]\Ra \beta}{(\mrm{As_1})}
        \quad
        \frac{\Gamma[(\Delta_1,\Delta_2),\Delta_3]\Ra \beta}{\Gamma[\Delta_1,(\Delta_2,\Delta_3)]\Ra \beta}{(\mrm{As_2})}
        \]
        \end{itemize}
        }
    \end{Definition}
    
    A sequent $\Gamma\Ra \beta$ is provable in $\msf{Gb}$, notation $\vdash_{\msf{Gb}}\Gamma\Ra \beta$, if there is a derivation of $\Gamma\Ra \beta$ in $\msf{Gb}$. We write $\vdash_{\msf{Gb}}\alpha\Leftrightarrow \beta$ if $\vdash_{\msf{Gb}}\alpha\Ra  \beta$ and $\vdash_{\msf{Gb}}\beta\Ra \alpha$. For $,$ and $;$ are commutative and associative, the rules of exchange and associativity are all admissible in $\msf{Gb}$. Hereafter we usually skip the applications of those rules in the derivations.
    
    \begin{Theorem}[Soundness and Completeness]\label{complete2}
    $\msf{qG}$ is sound and complete with respect to $cb\mc{RL^-}$s.
    \end{Theorem}   
    
    \begin{Theorem}[Cut-Elimination]\label{ce}
    $\vdash_{\msf{Gb}}\Gamma\Ra \beta$ iff $\vdash_{\msf{Gb}}\Gamma\Ra \beta$ without any application of $(\mrm{Cut})$ rule.
    \end{Theorem}
    
    \begin{proof}
    The proof is analogous to Theorem \ref{complete}.
    \end{proof}
    
    \begin{proof}
    Assume that there is a subderivation of $\Gamma\Ra\beta$ ended with an application of $(\mrm{Cut})$ as follows:
    \[
    \frac{\vdash\Delta\Ra\alpha \quad \vdash\Sigma[\alpha]\Ra\beta}{\vdash\Sigma[\Delta]\Ra\beta}{(\mrm{Cut})}
    \]
    We suffice to show that if $\Delta\Ra\alpha$ and $\Sigma[\alpha]\Ra\beta$ are both provable in $\msf{Gb}$ without any application of $(\mrm{Cut})$, then $\Sigma[\alpha]\Ra\beta$ is provable in $\msf{G'}$ without any application of $(\mrm{Cut})$. We proceed by induction on (I) the complexity of $(\mrm{Cut})$ formula $\alpha$. In each case, we proceed by induction on (II) the sum of the length of two premises of $(\mrm{Cut})$. Assume that $\Delta\Ra\alpha$ is obtained by $(R_l)$ and $\Sigma[\alpha]\Ra\beta$ is obtained by $(R_r)$. We refer the details to the standard cut-elimination proof.
    \begin{itemize} 
        \item[(1)] $\alpha$ is not introduced by $(R_l)$. We transform the derivation by first applying $(\mrm{Cut})$ to premises of $(R_l)$ and $\Sigma[\alpha]\Ra\beta$. After that, we apply $(R_l)$ to the resulting sequent. Take $(\vee\mrm{L})$ as an example to interpret this. The remaining cases can be treated similarly.\\
        $(R_l)$ is $(\vee\mrm{L})$. Then $\Delta=\Gamma[\alpha\vee\beta]$, $\alpha=\gamma$. Suppose the derivation ends with:
            \begin{prooftree}
                \AxiomC{$\Gamma[\alpha]\Ra\gamma$}
                \AxiomC{$\Gamma[\beta]\Ra\gamma$}
                \RightLabel{ $(\vee\mrm{L})$}
                \BinaryInfC{$\Gamma[\alpha\vee\beta]\Ra\gamma$}
                \AxiomC{$\Sigma[\gamma]\Ra\theta$}
                \RightLabel{ $(\mrm{Cut})$}
                \BinaryInfC{$\Sigma[\Gamma[\alpha\vee\beta]]\Ra\theta$}
            \end{prooftree}
        can be transformed into
            \begin{prooftree}
                \AxiomC{$\Gamma[\alpha]\Ra\gamma$}
                \AxiomC{$\Sigma[\gamma]\Ra\theta$}
                \RightLabel{ $(\mrm{Cut})$}
                \BinaryInfC{$\Sigma[\Gamma[\alpha]]\Ra\theta$}
                \AxiomC{$\Gamma[\beta]\Ra\gamma$}
                \AxiomC{$\Sigma[\gamma]\Ra\theta$}
                \RightLabel{ $(\mrm{Cut})$}
                \BinaryInfC{$\Sigma[\Gamma[\beta]]\Ra\theta$}
                \RightLabel{ $(\vee\mrm{L})$}
                \BinaryInfC{$\Sigma[\Gamma[\alpha\vee\beta]]\Ra\theta$}
            \end{prooftree}
        Thus the applications of (Cut) in the premises have lower lengths. Hence by induction hypothesis (II), the claim holds.\\
        We specially consider the case that $(R_l)$ is $(\mrm{R}-\bot)$, suppose the derivation ends with:
            \begin{prooftree}
                \AxiomC{$(\Delta_1,\Delta_2);\Delta_3\Ra\bot$}
                \RightLabel{ $(\mrm{R}-\bot)$}
                \UnaryInfC{$(\Delta_1,\Delta_3);\Delta_2\Ra\bot$}
                \AxiomC{$\Sigma[\bot]\Ra\beta$}
                \RightLabel{ $(\mrm{Cut})$}
                \BinaryInfC{$\Sigma[(\Delta_1,\Delta_3);\Delta_2]\Ra\beta$}
            \end{prooftree}
        can be transformed into:
            \begin{prooftree}
                \AxiomC{$(\Delta_1,\Delta_2);\Delta_3\Ra\bot$}
                \RightLabel{ $(\mrm{R}-\bot)$}
                \UnaryInfC{$(\Delta_1,\Delta_3);\Delta_2\Ra\bot$}
                \RightLabel{ $(\bot)$}
                \UnaryInfC{$\Sigma[(\Delta_1,\Delta_3);\Delta_2]\Ra\beta$}
            \end{prooftree}
        It is obviously cut-free, then the claim holds.
        
        \item[(2)] $\alpha$ is introduced by $(R_l)$ only. We transform the derivation by first applying $(\mrm{Cut})$ to the premise of $(R_r)$ and $\Delta\Ra\alpha$. After that, we apply $(R_r)$ to the resulting sequent. Take $(\bsl\mrm{L})$ as an example to interpret this. The remaining cases can be treated similarly.
        $(R_r)$ is $(\bsl\mrm{L})$, then $\Sigma[\alpha]=\Gamma[\Theta[\alpha],\beta\bsl\gamma]$, $\beta=\theta$. Suppose the derivation ends with:
            \begin{prooftree}
                \AxiomC{$\Delta\Ra\alpha$}
                \AxiomC{$\Theta[\alpha]\Ra\beta$}
                \AxiomC{$\Gamma[\gamma]\Ra\theta$}
                \RightLabel{ $(\bsl\mrm{L})$}
                \BinaryInfC{$\Gamma[\Theta[\alpha],\beta\bsl\gamma]\Ra\theta$}
                \RightLabel{ $(\mrm{Cut})$}
                \BinaryInfC{$\Gamma[\Theta[\Delta],\beta\bsl\gamma]\Ra\theta$}
            \end{prooftree}
        can be transformed into
            \begin{prooftree}
                \AxiomC{$\Delta\Ra\alpha$}
                \AxiomC{$\Theta[\alpha]\Ra\beta$}
                \RightLabel{ $(\mrm{Cut})$}
                \BinaryInfC{$\Theta[\Delta]\Ra\beta$}
                \AxiomC{$\Gamma[\gamma]\Ra\theta$}
                \RightLabel{ $(\bsl\mrm{L})$}
                \BinaryInfC{$\Gamma[\Theta[\Delta],\beta\bsl\gamma]\Ra\theta$}
            \end{prooftree}
        Thus the new application of $(\mrm{Cut})$ has a lower length of its premise. By induction hypothesis (II), the claim holds.\\
        We specially consider the case that $(R_r)$ is $(\mrm{R}-\bot)$, if $\alpha$ is in $\Delta_1$, suppose the derivation ends with:
            \begin{prooftree}
                \AxiomC{$\Delta\Ra\alpha$}
                \AxiomC{$(\Delta_1[\alpha],\Delta_2);\Delta_3\Ra\bot$}
                \RightLabel{ $(\mrm{R}-\bot)$}
                \UnaryInfC{$(\Delta_1[\alpha],\Delta_3);\Delta_2\Ra\bot$}
                \RightLabel{ $(\mrm{Cut})$}
                \BinaryInfC{$(\Delta_1[\Delta],\Delta_3);\Delta_2\Ra\bot$}
            \end{prooftree}
        can be transformed into:
            \begin{prooftree}
                \AxiomC{$\Delta\Ra\alpha$}
                \AxiomC{$(\Delta_1[\alpha],\Delta_2);\Delta_3\Ra\bot$}
                \RightLabel{ $(\mrm{Cut})$}
                \BinaryInfC{$(\Delta_1[\Delta],\Delta_2);\Delta_3\Ra\bot$}
                \RightLabel{ $(\mrm{R}-\bot)$}
                \UnaryInfC{$(\Delta_1[\Delta],\Delta_3);\Delta_2\Ra\bot$}
            \end{prooftree}
        Thus the new applications of $(\mrm{Cut})$ have a lower length of its right premise. By induction hypothesis (II), the claim holds. Other cases are very similar, we omit them.
        \item[(3)] $\alpha$ is introduced in both premises.  We transform the derivation by first applying $(\mrm{Cut})$ to the premise of $(R_l)$ and $(R_r)$. After that, we apply $(\mrm{Cut})$ to the resulting sequent. Take $(R_l)=(\land\mrm{R})$, $(R_r)=(\land\mrm{L})$ and $(R_l)=(\bsl\mrm{R})$, $(R_r)=(\bsl\mrm{L})$ as examples to interpret this. The remaining cases can be treated similarly.
        \\
        $(R_l)$ is $(\land\mrm{R})$, then $\Delta=\Gamma$, $\alpha=\alpha\land\beta$. Suppose the derivation ends with:
            \begin{prooftree}
                \AxiomC{$\Ga\Ra\al$}
                \AxiomC{$\Ga\Ra\be$}
                \RightLabel{$(\land R)$}
                \BinaryInfC{$Ga\Ra\al\land\be$}
                \AxiomC{$\Theta[\al]\Ra\ga$}
                \RightLabel{$(\land L)$}
                \UnaryInfC{$\Theta[\al\land\be]\Ra\ga$}
                \RightLabel{(Cut)}
                \BinaryInfC{$\Theta[\Ga]\Ra\ga$}
            \end{prooftree}
        can be transformed into
            \begin{prooftree}
                \AxiomC{$\Ga\Ra\al$}
                \AxiomC{$\Theta[\al]\Ra\ga$}
                \RightLabel{(Cut)}
                \BinaryInfC{$\Theta[\Ga]\Ra\ga$}
            \end{prooftree}
        $(R_l)$ is $(\bsl\mrm{R})$, then $\Delta=\Gamma$, $\alpha=\alpha\bsl\beta$. Suppose the derivation ends with:
            \begin{prooftree}
                \AxiomC{$\alpha,\Gamma\Ra\beta$}
                \RightLabel{ $(\bsl\mrm{R})$}
                \UnaryInfC{$\Gamma\Ra\alpha\bsl\beta$}
                \AxiomC{$\Theta\Ra\alpha$}
                \AxiomC{$\Sigma[\beta]\Ra\gamma$}
                \RightLabel{ $(\bsl\mrm{L})$}
                \BinaryInfC{$\Sigma[\Theta,\alpha\bsl\beta]\Ra\gamma$}
                \RightLabel{ $(\mrm{Cut})$}
                \BinaryInfC{$\Sigma[\Theta,\Gamma]\Ra\gamma$}
            \end{prooftree}
        can be transformed into
                \begin{prooftree}
                \AxiomC{$\alpha,\Gamma\Ra\beta$}
                \AxiomC{$\Sigma[\beta]\Ra\gamma$}
                \RightLabel{ $(\mrm{Cut})$}
                \BinaryInfC{$\Sigma[\alpha,\Gamma]\Ra\gamma$}
                \AxiomC{$\Theta\Ra\alpha$}
                \RightLabel{ $(\mrm{Cut})$}
                \BinaryInfC{$\Sigma[\Theta,\Gamma]\Ra\gamma$}
            \end{prooftree}
    \end{itemize}
    Thus the complexity of the (Cut) formula in the first and second applications of (Cut) is lower than the original one. Hence by induction hypothesis (I), the claim holds.
    \end{proof}

    \begin{Corollary}[Subformula Property]
    If $\vdash_{\msf{Gb}}\Ga\Ra\al$, then there exists a derivation of $\Gamma\Ra\alpha$ in $\msf{Gb}$ such that all formulas appearing in the proof are subformulas of formulae appearing in $\Gamma\Ra\alpha$.
    \end{Corollary}

    \begin{proof}
    It immediately follows from Theorem \ref{ce}.
    \end{proof}

    \begin{Definition}
    We define the length of a sequent $\Gamma\Ra\alpha$ as the number of formulas in it, denoted by $l(\Gamma\Ra\alpha)$.
    \end{Definition}
    
    \begin{Theorem}[Decidability]\label{decidable}
    $\msf{Gb}$ is decidable.
    \end{Theorem}

    \begin{proof}
    From all rules in $\msf{Gb}$, we can see that the length of premise sequent $l(\Gamma\Ra\alpha)$ is less than or equal to that of conclusion sequent except for the (Cut) rule. Then by K\"{o}nig's Lemma (cf. \cite{konig1927schlussweise}) and Theorem \ref{ce}, in any derivation with endsequent $\Gamma\Ra\alpha$ there is all lower length of  sequents, which means the search space for $\Gamma\Ra\alpha$ is finite. Therefore, we conclude that $\msf{Gb}$ is decidable.
    \end{proof}

    \begin{Theorem}[Conservative Extension]\label{conservativextension} for any $\msf{qG}$ sequent $\alpha\Ra \beta$, 
    $\vdash_{\msf{qG}}\alpha\Ra \beta$ iff $\vdash_{\msf{Gb}}\alpha\Ra\beta$.
    \end{Theorem}

    \begin{proof}
    By Theorem \ref{complete}, $\vdash_{\msf{qG}}\alpha\Ra \beta$ iff  $\vDash_{qInc\mc{RL^-}} \alpha\Ra \beta$. While by Theorem \ref{complete2}, $\vdash_{\msf{Gb}}\alpha\Ra \beta$ iff  $\vDash_{cb\mc{RL^-}} \alpha\Ra \beta$. Hence one suffices to show  $\vDash_{qInc\mc{RL^-}} \alpha\Ra \beta$ iff $\vDash_{cb\mc{RL^-}} \alpha\Ra \beta$. The left to the right direction is easy since every $cb\mc{RL^-}$ is a $qInc\mc{RL^-}$. Conversely assume that  $\not\vDash_{qInc\mc{RL^-}} \alpha\Ra \beta$, then by Theorem \ref{expand}, one obtains a $cb\mc{RL^-}$ $\mc{M}$ such that $\not\vDash_{\mc{M}}\alpha\Ra \beta$. Therefore $\not\vDash_{cb\mc{RL^-}}\alpha\Ra \beta$. This completes the proof.
    \end{proof}

    \begin{Theorem}[Decidability]\label{decidable2}
    $\msf{qG}$ is decidable.
    \end{Theorem}

    \begin{proof}
    Since $\msf{qG}$ is a conservative extension of $\msf{Gb}$ by Theorem \ref{conservativextension} and $\msf{Gb}$ is decidable by Theorem \ref{decidable}, $\msf{qG}$ is decidable
    \end{proof}

    \begin{Corollary}
    $\msf{G}$ is decidable
    \end{Corollary}
    
    \begin{proof}
    Immediately follows from  Lemma \ref{lemma1} and Theorem \ref{decidable2}.
   \end{proof}

    \begin{Corollary}
     InFL$^-_e$ is decidable.
    \end{Corollary}

    If one omits the rule $({\ast}\mrm{L})$ and $({\ast}\mrm{R})$ and replacing $({\ra}\mrm{L})$,$({\ra}\mrm{R})$ by the following negation rules,
        \[
        \frac{ \Delta\Ra\alpha }{\Delta;\neg \alpha\Ra \bot}{(\neg \mrm{L})}
        \quad
        \frac{\alpha;\Gamma\Ra \bot}{\Gamma\Ra\neg\al}{(\neg \mrm{R})}
        \]
    then one obtains a sequent calculus for qInFL$^-_e$ denoted by ($\msf{qG_c}$). Clearly ($\neg$L) and ($\neg$R) are instance of$({\ra}\mrm{L})$,$({\ra}\mrm{R})$ respectively. Hence the cut-elimination and subformula property preserve in $\msf{qG_c}$. By adding the following two rules to the cut-free system of $\msf{qG_c}$, one obtains a cut-free system of $\msf{G}$ denoted by $\msf{G_c}$.  
        \[
        \frac{ \Gamma[ko(\alpha)]\Ra\beta }{\Gamma[\alpha]\Ra \beta}{(\mrm{koL})}
        \quad
        \frac{ \Gamma\Ra ko(\alpha) }{\Gamma\Ra \alpha}{(\mrm{koR})}
        \]

    \begin{Theorem}
    For any sequent $\alpha\Ra\beta$, if $\vdash_{\msf{G}}\alpha\Ra\beta$, then $\vdash_{\msf{G_c}} \alpha\Ra \beta$.
    \end{Theorem} 
    
    \begin{proof}
    Let $\vdash_{\msf{G}} \alpha\Ra \beta$. Then $\vdash_{\msf{qG}}ko(\alpha)\Ra ko(\beta)$. Thus $\vdash_{\msf{Gb}} ko(\alpha)\Ra ko(\beta)$. Hence  $\vdash_{\msf{qG_c}} ko(\alpha)\Ra ko(\beta)$. Therefore $\vdash_{\msf{G_c}} ko(\alpha)\Ra ko(\beta)$. Consequently by ($\mrm{ko}$L) and  ($\mrm{ko}$R), $\vdash_{\msf{G_c}} \alpha\Ra \beta$.
    \end{proof}
    
    Although an additional structural operator $;$ is contained in system $\msf{G_c}$, $\msf{G_c}$ can be viewed as a cut-free sequent calculus of InFL$^-_e$ with $\top,\bot$. This is because any InFL$^-_e$ valid sequent is provable in $\msf{G_c}$. This idea is inspired by the cut-free sequent calculus for product-free Lambek calculus. 

\section{Conclusion and Future Work}

    We study the commutative residuated lattice-ordered semigroup with involutive negation. Simple sequent calculi for the algebras are discussed. The soundness and completeness of the logic InFL$^-_e$ are proved. A standard double negation translation result is established. Further, the decidability results of the corresponding logics are studied. Finally, we design a kind of cut-free sequent calculus for InFL$^-_e$ enriched with constant $\top, \bot$. Our results can be extended to minimal negation under commutative residuated lattice-ordered semigroup. A further interesting question is that can we adapt these results to commutative residuated lattice-ordered semigroup with De morgan negations.

\section*{Acknowledgments}

TBA.

\section*{Declarations}

\begin{itemize}
\item Funding: This research was funded by basic scientific research business cost project of central universities, grant number 2072021107.
\item Conflict of interest: The authors declare no conflict of interest.
\item Ethics approval: Not applicable.
\item Consent to participate: Not applicable.
\item Consent for publication: Not applicable.
\item Availability of data and materials: Not applicable.
\item Code availability: Not applicable.
\item Authors' contributions: TBA.
\end{itemize}






\bibliography{sn-bibliography}


\end{document}